\documentclass[12pt]{amsart}
\usepackage{amsmath,amsthm,amsfonts,latexsym,amssymb}
\usepackage{graphicx}
\usepackage{mathrsfs}
\input xy
\xyoption{all}

\makeatletter
\providecommand*{\Dashv}{%
  \mathrel{%
    \mathpalette\@Dashv\vDash
  }%
}
\newcommand*{\@Dashv}[2]{%
  \reflectbox{$\m@th#1#2$}%
}
\makeatother

\theoremstyle{theorem}
\newtheorem{theorem}{Theorem}
\newtheorem{corollary}{Corollary}
\newtheorem{lemma}{Lemma}
\newtheorem{proposition}{Proposition}

\long\def\symbolfootnote[#1]#2{\begingroup%
\def\thefootnote{\fnsymbol{footnote}}\footnote[#1]{#2}\endgroup}

\theoremstyle{definition}
\newtheorem{definition}{Definition}
\newtheorem{remark}{Remark}
\newtheorem{question}{Question}
\newtheorem{fact}{Fact}
\newtheorem{observation}{Observation}

\newtheorem{conjecture}{Conjecture}

\numberwithin{equation}{section}
\newcommand{\so}{\mathbf{s}}
\newcommand{\ra}{\mathbf{r}}
\newcommand{\Spec}{\mathrm{Spec}}
\newcommand{\pth}{\mathrm{Path}}
\newcommand{\AT}{\mathcal{A}}
\newcommand{\N}{\mathbb{N}}
\newcommand{\Q}{\mathbb{Q}}
\newcommand{\R}{\mathbb{R}}

\usepackage{fancyhdr}
\begin{document}

\title[An infinite cardinal-valued Krull dimension for rings]
{An infinite cardinal-valued Krull dimension for rings}

\author{K. Alan Loper}
\address[K. Alan Loper]{Department of Mathematics\\
The Ohio State University\\
Newark, OH 43055\\
USA}
\email{lopera@math.ohio-state.edu}

\author{Zachary Mesyan}
\address[Zachary Mesyan]{Department of Mathematics\\
University of Colorado\\
Colorado Springs, CO 80918\\
USA}
\email{zmesyan@uccs.edu}

\author{Greg Oman}
\address[Greg Oman]{Department of Mathematics\\
University of Colorado\\
Colorado Springs, CO 80918\\
USA}
\email{goman@uccs.edu}

\symbolfootnote[0]{\emph{2010 Mathematics Subject Classification.} Primary: 13A15; Secondary: 03E10.}
\symbolfootnote[0]{\emph{Key Words and Phrases.} cofinality, Generalized Continuum Hypothesis, Krull dimension, Leavitt path algebra, Noetherian ring, ordered set, polynomial ring, valuation ring.}

\begin{abstract} We define and study two generalizations of the Krull dimension for rings, which can assume cardinal number values of arbitrary size. The first, which we call the \emph{cardinal Krull dimension}, is the supremum of the cardinalities of chains of prime ideals in the ring. The second, which we call the \emph{strong cardinal Krull dimension}, is a slight strengthening of the first. Our main objective is to address the following question: for which cardinal pairs $(\kappa,\lambda)$ does there exist a ring of cardinality $\kappa$ and (strong) cardinal Krull dimension $\lambda$? Relying on results from the literature, we answer this question completely in the case where $\kappa\geq\lambda$. We also give several constructions, utilizing valuation rings, polynomial rings, and Leavitt path algebras, of rings having cardinality $\kappa$ and (strong) cardinal Krull dimension $\lambda>\kappa$. The exact values of $\kappa$ and $\lambda$ that occur in this situation depend on set-theoretic assumptions. \end{abstract}

\maketitle

\section{Introduction}

Before motivating this article, we fix some terminology. All rings will be assumed to be associative (though not commutative unless stated explicitly), with $1\neq 0$ unless noted otherwise, and all extensions of unital rings will be unital. 

Recall that a collection $\mathcal{C}$ of sets is a \emph{chain} if for any $A,B\in\mathcal{C}$, either $A\subseteq B$ or $B\subseteq A$. Next, if $R$ is a ring and $n$ is a non-negative integer, then a chain $P_0\subsetneq P_1\subsetneq\cdots P_n$ of prime ideals of $R$ is said to have \emph{length} $n$.\footnote{The length measures the number of \emph{containments}, not the number of elements in the chain.} Now let $\kappa$ be an infinite cardinal and suppose that $\mathcal{C}$ is a chain of $\kappa$ many prime ideals of $R$. Then we say that $\mathcal{C}$ has \emph{length} $\kappa$. If $\mathcal{C}$ is any nonempty chain of prime ideals of $R$, then we denote the length of $\mathcal{C}$ by $\ell(\mathcal{C})$. 

The \emph{Krull dimension} of a commutative ring $R$ is often defined to be $n$ if $n$ is the largest non-negative integer for which there is a chain of prime ideals of $R$ of length $n$. If no such $n$ exists, it is common for the Krull dimension of $R$ to be defined to be simply ``$\infty$" (see \cite{RG}, pp. 105--106 and \cite{LAM}, p. 65). We introduce a more stringent definition.

\begin{definition}\label{def1} Let $R$ be a ring, and let $\kappa:=\sup\{\ell(\mathcal{C})\mid\mathcal{C}~\textrm{a~nonempty~chain~of~prime~ideals~of}~R\}$. We call the cardinal $\kappa$ the \emph{cardinal Krull dimension} of $R$, and we write $\mathrm{c.dim}(R)=\kappa$. \end{definition}

\begin{remark}\label{zremark} Various other infinite-valued generalizations of the Krull dimension for commutative rings appear in the literature. A particularly well-studied one, called the \emph{classical Krull dimension}, was introduced by Krause in \cite{Krause}, and assumes ordinal numbers as values; see also \cite{KG}, Chapter 12, for a discussion. Another ordinal-theoretic generalization, called the \emph{little classical Krull dimension}, was introduced by Gordon and Robson in \cite{JR}. Their definition is similar to ours, except only well-ordered chains of prime ideals are considered, and the ``length" of a chain is defined in terms of the order type of the chain. 

We also refer the reader to \cite{KG}, pp. 221 and 239, for a historical overview of the Krull dimension. To help with expositional clarity, throughout this article, we shall always append one of two phrases (the second of which to be defined shortly) to ``Krull dimension" when there is danger of confusion.\end{remark}

\noindent We pause to take note of the following immediate consequence of the previous definition:

\begin{observation}\label{obs1} Let $R$ be a ring. Then $\mathrm{c.dim(R)}\leq2^{|R|}$. \end{observation}

\noindent The concept of finite Krull dimension (that is, the supremum of the lengths of chains of prime ideals is finite) is ubiquitous both in algebraic geometry and the ideal theory of commutative rings. As a sampling, we record several well-known results (we refer the reader to the texts \cite{DE}, \cite{RG}, and \cite{RH} for details).

\begin{fact}\label{fact1} Let $R$ be a commutative ring. Then the following hold:
\begin{enumerate}
\item $R$ is Artinian if and only if $R$ is Noetherian with Krull dimension $0$. 
\item If $R$ is Noetherian of Krull dimension $m\in\mathbb{N}$ and $X_1,\ldots,X_n$ are indeterminates, then the polynomial ring $R[X_1,\ldots,X_n]$ has Krull dimension $m+n$. 
\item $R$ is a Dedekind domain (which is not a field) if and only if $R$ is a Noetherian, integrally closed domain with Krull dimension $1$. 
\item Suppose $R$ is a domain which is not a field with quotient field $K$. Then there are no rings properly between $R$ and $K$ if and only if $R$ is a valuation ring with Krull dimension $1$.
\end{enumerate}
\end{fact}

A natural question is whether the cardinal Krull dimension of a ring $R$ can always be realized by a chain of prime ideals. In other words, if $R$ is a ring and $\mathrm{c.dim}(R)=\kappa$, is it the case that there exists a chain of prime ideals of $R$ of length $\kappa$? The answer turns out to be ``no". To see this, recall that every chain of prime ideals of a commutative Noetherian ring is finite (This follows from the well-known fact that every such ring satisfies the descending chain condition on prime ideals. Thus a nonempty subset of a chain of prime ideals in a commutative Noetherian ring must have a least and a greatest element. Furthermore, a totally ordered set with the property that every nonempty subset has both a greatest and a least element is necessarily finite.). In contrast to this fact, Nagata showed in 1962 (\cite{MN}) that there exist commutative Noetherian rings with chains of prime ideals of arbitrarily long finite length. Thus there exist rings $R$ for which $\mathrm{c.dim}(R)=\aleph_0$, yet there is no chain of prime ideals of $R$ of length $\aleph_0$. Later in the paper we establish, for every limit cardinal $\kappa$, the existence of a ring of cardinal Krull dimension $\kappa$ but without a chain of prime ideals of length $\kappa$ (see Proposition \ref{berry1}). In light of these facts, we make the following definition:

\begin{definition}\label{def00} Let $R$ be a ring. Say that $R$ has \emph{strong cardinal Krull dimension} $\kappa$ provided $\mathrm{c.dim}(R)=\kappa$ and there exists a chain of prime ideals of $R$ of length $\kappa$. If this is the case, we write $\mathrm{sc.dim}(R)=\kappa$. \end{definition}

Our comments above demonstrate that while the cardinal Krull dimension of any ring always exists, there are rings which do not have strong cardinal Krull dimension. Nevertheless, often the two notions coincide. We record observations toward this end below. The proofs are straightforward; as such, they are omitted.

\begin{observation}\label{obs0} Let $R$ be a ring. Then the following hold:
\begin{enumerate}
\item If $R$ has finite cardinal Krull dimension, then $\mathrm{sc.dim}(R)$ exists. More generally, 
\item if $\mathrm{c.dim}(R)$ is a successor cardinal, then $\mathrm{sc.dim}(R)$ exists. 
\item If the prime ideals of $R$ form a chain with respect to set inclusion, then $\mathrm{sc.dim}(R)$ exists. 
\end{enumerate}
\end{observation}

Aside from Nagata's example, another impetus for the definitions and results presented in this note stems from the recent paper of Kang and Park (\cite{BK}) which proves that the power series ring (in one variable) over a nondiscrete valuation ring has uncountable chains of prime ideals. Additionally, we mention the dissertation of Bailey (\cite{AB}), in which he shows that it is consistent with ZFC that for every infinite cardinal $\kappa$, there is a valuation ring of cardinality $\kappa$ with a chain of $2^\kappa$ prime ideals. We record the related observation communicated to the first author by Jim Coykendall below. The result seems to have been known since antiquity and can be found, for example, on p. 1501 of \cite{DD2}. The proof we give is the standard one.

\begin{proposition}[folklore]\label{prop1} Let $F$ be a countable field, and let $X_1,X_2,X_3,\ldots$ be indeterminates. For every positive integer $n$, set $R_n:=F[X_1,\ldots,X_n]$, and let $R_\infty:=F[X_1,X_2,X_3,\ldots]$. Then 

\begin{enumerate}
\item $\mathrm{sc.dim}(R_n)=n$ for every integer $n>0$, yet 
\item $\mathrm{sc.dim}(R_\infty)=2^{\aleph_0}$. 
\end{enumerate}
\end{proposition}

\begin{proof} Assume that $F$, $R_n$, and $R_\infty$ are as stated above. 

\vspace{.1 in}

(1) The assertion follows immediately from (1) and (2) of Fact \ref{fact1} along with (1) of Observation \ref{obs0}. 

\vspace{.1 in}

(2) Because $|\mathbb{Z}^+|=|\mathbb{Q}|$, we may re-enumerate the variables with respect to $\mathbb{Q}$. Hence we may assume that $R_\infty=F[X_q\mid q\in\mathbb{Q}]$. Now, for every real number $r$, let $P_r:=\langle X_i\mid i\in(-\infty,r)\cap\mathbb{Q}\rangle$. It is clear that each $P_r$ is a prime ideal of $R_\infty$ and that, moreover, $P_r\subsetneq P_s$ whenever $r,s\in\mathbb{R}$ with $r<s$. Setting $\mathcal{C}:=\{P_r\mid r\in\mathbb{R}\}$, we see that 

\begin{equation}\label{eq1}
\mathcal{C}~\mathrm{is~a~chain~of~prime~ideals~of}~R_\infty~\mathrm{of~cardinality}~2^{\aleph_0}.
\end{equation} 

\noindent Applying Observation \ref{obs1} and \eqref{eq1}, we have $2^{\aleph_0}\leq\mathrm{c.dim}(R_\infty)\leq2^{|R_\infty|}=2^{\aleph_0}$. Thus equality holds throughout. It is immediate from \eqref{eq1} that $R$ has strong cardinal Krull dimension. \end{proof}

As final motivation for the results presented in this paper, we mention the interesting work of Dobbs, Heitmann, and Mullins (\cite{DD0}-\cite{DD3}) who study lengths of chains of various algebraic structures including subspaces of a vector space, overrings, and field extensions. They work mostly in ZFC. In this note, we expand the set-theoretic focus a bit; in particular, we establish some results which are independent of ZFC. Many such results translate to chains of other algebraic objects. We now state the central question addressed in this article:

\begin{question}\label{qu1} For which cardinal pairs $(\kappa,\lambda)$ does there exist a ring $R$ of cardinality $\kappa$ and  (strong) cardinal Krull dimension $\lambda$? \end{question}

The outline of the paper is as follows. The next section gives a terse treatment of some fundamental order theory and set theory utilized throughout the paper. In Section 3, we explore Question \ref{qu1} in the context of valuation rings. Section 4 is devoted to a similar analysis, but with valuation rings replaced with polynomial rings. Section 5 continues the investigation of Question \ref{qu1} relative to Leavitt path algebras. More specifically, we show that if $\kappa \geq \lambda$, then there
exists a ring with cardinality $\kappa$ and strong cardinal Krull dimension
$\lambda$ if and only if $\kappa \geq 2$ and either $\kappa \geq \omega$
or $\lambda =0$ (Theorem \ref{thm1}). We also construct, for any infinite cardinal
$\kappa$, a valuation ring (Corollary \ref{sun}), a polynomial ring (Corollary \ref{thm2}),
and a Leavitt path algebra (Theorem \ref{thm7}) having cardinality $\kappa$ and
cardinal Krull dimension $\lambda$ strictly larger than $\kappa$. Most of the
rest of the paper discusses what values of
$\lambda$ and $\kappa$ are possible in this situation, and how that
depends on set-theoretic assumptions (see Theorem \ref{yetanother}, Corollary \ref{ttb},
Corollary \ref{cccc}, and Proposition \ref{forceprop}).

\section{Preliminaries}

In this terse section, we present several definitions and results of which we shall make frequent use.

\subsection{Order Theory} 

Let $P$ be a set. A \emph{preorder} on $P$ is a binary relation $\leq$ on $P$ which is reflexive and transitive. If in addition, $\leq$ is antisymmetric, then we say that $\leq$ is a \emph{partial order} on $P$. In this case, $(P,\leq)$ is called a \emph{partially ordered set} or \emph{poset}. If in addition, for any $x,y\in P$, either $x\leq y$ or $y\leq x$, then we say that $\leq$ is a \emph{total order} or \emph{linear order} on $P$, and call $(P,\leq)$ a \emph{linearly ordered set} or a \emph{chain}. We call $\leq$ a \emph{well-order} on $P$ if, in addition to being linear, every nonempty subset of $P$ has a $\leq$-least element. Finally, suppose that $(L,\leq)$ is a linearly ordered set and that $D\subseteq L$. Then $D$ is \emph{dense} in $L$ provided that for every $x,y\in L$ with $x<y$, there exists some $d\in D$ such that $x<d<y$. 

If $\prec$ is a binary relation on $P$ which is irreflexive and transitive, then $\prec$ is often called a \emph{strict} partial order on $P$. Observe that if $\prec$ is a strict partial order on $P$, then the binary relation $\preceq$ defined on $P$ by $x\preceq y$ if and only if $x=y$ or $x\prec y$ is a partial order on $P$. Analogously, if $\preceq$ is a partial order on $P$, then the binary relation $\prec$ defined on $P$ by $x\prec y$ if and only if $x\preceq y$ and $x\neq y$ is a strict partial order on $P$. Similar remarks apply in case $\preceq$ is a linear or well-order on $P$. 

\subsection{Set Theory}\label{sub}

Recall that if $s$ is any set, then the set membership relation $\in$ induces a binary relation $\in_s$ on $s$ defined by $\in_s:=\{(a,b)\in s\times s\mid a\in b\}$. The relation $\in_s$ is called the \emph{epsilon relation} on $s$. A set $s$ is \emph{transitive} provided if $a\in b\in s$, then also $a\in s$. A transitive set $s$ which is (strictly) well-ordered by $\in_s$ is called an \emph{ordinal}.  Let $ORD$ denote the proper class of ordinal numbers. The membership relation induces a well-order on $ORD$; when there is no ambiguity, this relation will be denoted by the more commonplace $<$. 

The axiom of choice (along with the other axioms of ZFC) implies that every set is equinumerous with an ordinal. If $x$ is any set, then the least ordinal $\alpha$ equinumerous with $x$ (relative to the membership relation just described) is called the \emph{cardinality} of $x$; we say that $\alpha$ is the cardinal number of $x$, and write $|x|=\alpha$. An ordinal $\alpha$ is a \emph{cardinal number} if $\alpha=|x|$ for some set $x$. Let $\kappa$ be an infinite cardinal. The \emph{cofinality} of $\kappa$, denoted $\mathrm{cf}(\kappa)$, is the least cardinal $\lambda$ such that $\kappa$ is the union of $\lambda$ cardinals, each smaller than $\kappa$. We say that $\kappa$ is \emph{regular} if $\mathrm{cf}(\kappa)=\kappa$. If $\kappa$ is not regular, then $\kappa$ is called \emph{singular}. 

Now let $\alpha\in ORD$ be arbitrary. Then $\alpha$ is a \emph{limit ordinal} provided that for every $i<\alpha$, there is an ordinal $j$ such that $i<j<\alpha$. The ordinal $\alpha$ is a \emph{successor ordinal} if $\alpha=i+1$ for some ordinal $i$, where $i+1:=i\cup\{i\}$ (which is the least ordinal greater than $i$). It is easy to prove that every ordinal number is either a limit or a successor ordinal. Analogously, a cardinal number $\kappa$ is a \emph{successor cardinal} if there is a cardinal number $\beta$ such that $\kappa$ is the least cardinal larger than $\beta$. In this case, we write $\kappa=\beta^+$. If there is no such $\beta$, then $\kappa$ is a \emph{limit cardinal}.\footnote{It is easy to see that every infinite cardinal is a limit ordinal, but not every infinite cardinal is a limit cardinal.}

Next, let $\alpha$ and $\beta$ be cardinals and let $X$ and $Y$ be disjoint sets of cardinality $\alpha$ and $\beta$, respectively. The sum $\alpha+\beta$ of $\alpha$ and $\beta$ is defined to be $|X\cup Y|$. Moreover, the product $\alpha\cdot\beta$ is $|X\times Y|$, and the power $\alpha^\beta$ is defined by $\alpha^\beta:=|X^Y|$, where $X^Y:=\{f\mid f\colon Y\rightarrow X$ is a function$\}$. We now recall several elementary results on cardinal arithmetic which will be useful throughout this paper:

\begin{fact}\label{cardfact} Let $\alpha$ and $\beta$ be cardinals with $\alpha$ nonzero and $\beta$ infinite. Then:
\begin{enumerate}
\item $\alpha+\beta=\max(\alpha,\beta)=\alpha\cdot\beta$,
\item$\beta^\beta=2^\beta$, and
\item $\beta$ is a limit ordinal. 
\end{enumerate} 
\end{fact} 

The \emph{Continuum Hypothesis} (CH) is the assertion that there is no cardinal number properly between $\aleph_0$ and $2^{\aleph_0}$. The \emph{Generalized Continuum Hypothesis} (GCH) is the claim that for every infinite cardinal $\kappa$, there is no cardinal  properly between $\kappa$ and $2^\kappa$. By work of G\"odel and Cohen, CH (respectively, GCH) cannot be proved nor refuted from the ZFC axioms, assuming the axioms are consistent. Finally, we mention the well-known texts \cite{TJ} and \cite{KK} as references for the elementary results presented above as well as for more advanced set theory.

\section{Valuation Rings}

Before presenting our main results, we set the stage by introducing some fundamentals of valuation theory. We remark that throughout, we adopt the model-theoretic/universal-algebraic convention of denoting ordered groups and, more generally, ordered sets with boldface capital letters and the ground set (universe) with non-boldface capital letters. However, we use a standard font for rings.

\subsection{Ordered Abelian Groups and Valuations}

An \emph{ordered abelian group} is a triple $\mathbf{G}:=(G,+,\leq)$ consisting of an abelian group $(G,+)$ and a translation-invariant total order $\leq$ on $G$. In other words, $\leq$ is a total order on $G$ and is such that for all $a,b,c\in G$: if $a\leq b$, then $a+c\leq b+c$. Suppose now that $\mathbf{I}:=(I,\leq)$ is a well-ordered set and that for every $i\in I$, $\mathbf{G}_i:=(G_i,+_i,\leq_i)$ is an ordered abelian group. The group $\bigoplus_{i\in I}\mathbf{G}_i$ denotes the external direct sum of the groups $\mathbf{G}_i$ with the \emph{lexicographic order} defined as follows. Suppose that $f,g\in\bigoplus_{i\in I}G_i$ are distinct, and choose the least $i\in I$ such that $f(i)\neq g(i)$. Then $f<g$ if $f(i)<_ig(i)$. 

If $\mathbf{G}$ is a totally ordered abelian group, then we may \emph{adjoin infinity} by declaring $g+\infty=\infty+g=\infty+\infty=\infty$ for any $g\in G$ (here, $\infty$ is any set which is not already a member of $G$). Moreover, we may extend $\leq$ to $G\cup\{\infty\}$ by defining $g\leq\infty$ for $g\in G$. Suppose now that $(G,+,\leq,\infty)$ is a totally ordered abelian group with infinity adjoined, and let $F$ be a field. A function $v\colon F\to G\cup\{\infty\}$ is said to be a \emph{valuation on $F$ with value group $\mathbf{G}$} if the following conditions hold:

\begin{enumerate}
\item $v(0)=\infty$,
\item $v(F\backslash\{0\})=G$,
\item $v(xy)=v(x)+v(y)$ for all $x,y\in F$, and
\item $v(x+y)\geq\min(v(x),v(y))$ for all $x,y\in F$. 
\end{enumerate}

\noindent If we set $V:=\{x\in F\mid v(x)\geq 0\}$, then one checks that $V$ is an integral domain whose lattice of ideals forms a totally ordered chain with respect to set inclusion. The commutative domain $V$ is called the \emph{valuation ring associated with $v$}. The next result is well-known:

\begin{fact}[Krull]\label{fact2} Let $\mathbf{G}$ be a totally ordered abelian group. Then there exists a field $F$ and a valuation $v$ on $F$ such that $\mathbf{G}$ is the value group of $v$. Moreover, if $\mathbf{G}$ is nontrivial, we may find $F$ having cardinality $|G|$ (see \cite{RG}, Corollary 18.5).\end{fact}

Next, let $\mathbf{G}$ be a totally ordered abelian group. For $g\in G$, we define the absolute value of $g$ in the natural way: $|g|:=g$ if $g\geq 0$ and $|g|:=-g$ otherwise. A subgroup $\mathbf{H}$ of $\mathbf{G}$ is called \emph{isolated} if for every $h\in H$, the interval $[-h,h]:=\{x\in G\mid |x|\leq|h|\}$ is contained in $H$. The \emph{rank} $\rho(\mathbf{G})$ of $\mathbf{G}$ is simply the order type of the collection of nontrivial isolated subgroups of $\mathbf{G}$, which is totally ordered by inclusion. The following is well-known; for a proof, see \cite{MK}, Result 1.

\begin{fact}\label{fact3} Suppose that $v$ is a valuation on a field $F$ with value group $\mathbf{G}$ and associated valuation ring $V$. Then there is a one-to-one inclusion-reversing correspondence between the collection of prime ideals of $V$ and the collection of isolated subgroups of $\mathbf{G}$. \end{fact}

We introduce a bit more notation and then conclude this section with a result from \cite{MK} to be invoked shortly. In what follows, if $\mathbf{L}:=(L,\leq)$ is a linearly ordered set, then set $\mathbf{L}':=(L'=L,\leq^{\text{op}})$, where $\leq^{\text{op}}$ is defined on $L$ by $x\leq^{\text{op}}y$ if and only if $y\leq x$. Next, suppose that $\mathbf{I}$ is a linearly ordered set and that for every $i\in I$, $\mathbf{S}_i:=(S_i,\leq_i)$ is also a linearly ordered set. Now define $\star_{i\in I}\mathbf{S}_i:=\{(i,s)\mid i\in I, s\in S_i\}$, endowed with the order $\leq^{\star}$ given by $(i,s)\leq^{\star}(j,t)$ if and only if $i<j$ or $i=j$ and $s\leq_it$. It is easy to see that $\leq^{\star}$ is a linear order on $\{(i,s)\mid i\in I, s\in S_i\}$; the structure $\star_{i\in I}\mathbf{S}_i$ is called the \emph{concatenation of the $\mathbf{S}_i$'s} (relative to $\mathbf{I}$).

\begin{theorem}[\cite{MK}, Theorem 1]\label{prop2} Let $\mathbf{I}$ be a well-ordered set, and suppose that for every $i\in I$, $\mathbf{G}_i$ is a totally ordered abelian group. Then $\rho(\bigoplus_{i\in I}\mathbf{G}_i)=\star_{i\in I'}\rho(\mathbf{G}_i)$. \label{prop2}  \end{theorem}

\subsection{The Downward-Restricted Version of Question \ref{qu1}}

The purpose of this subsection is to state and prove our first theorem, which resolves Question \ref{qu1} in the case $\kappa\geq\lambda$.

\begin{theorem}\label{thm1} Let $\kappa$ and $\lambda$ be cardinals with $\kappa\geq\lambda$. Then there exists a ring $R$ of cardinality $\kappa$ with cardinal Krull dimension $\lambda$ if and only if the following hold:

\begin{enumerate}
\item$\kappa\geq2$, and
\item if $\kappa<\omega$, then $\lambda=0$.
\end{enumerate}

\noindent Moreover, if $($1$)$ and $($2$)$ above hold, we may find such a ring $R$ with strong cardinal Krull dimension.

\end{theorem}

\begin{proof} Recall that for us, all rings are unital with $1\neq 0$. Moreover, every finite ring has Krull dimension $0$ as every prime ideal of a finite ring is maximal.\footnote{Of course, every finite ring is Artinian. It is well-known that every prime ideal of a commutative Artinian ring is maximal. If $R$ is not commutative, the result still holds; see \cite{LAM}, Exercise 10.4.} The necessity of (1) and (2) above is now clear. Conversely, suppose that $\kappa\geq\lambda$ are cardinals satisfying (1) and (2). Observe that if $\kappa<\omega$, then we may take $R:=\mathbb{Z}/\langle\kappa\rangle$. Now assume $\kappa\geq\omega$. If $\lambda=0$, then choose $R$ to be any field of cardinality $\kappa$. Suppose that $\lambda>0$. Note that the ordered abelian group $(\mathbb{Z},+,\leq)$ has no proper nontrivial isolated subgroups. Applying Theorem \ref{prop2}, it is easy to see that

\begin{equation}\label{eq3}
 \rho(\bigoplus_{\lambda}\mathbb{Z})\cong(\lambda,\in_\lambda^{\text{op}}). 
\end{equation}

\noindent Set $\mathbf{G}:=\bigoplus_\lambda\mathbb{Z}$ and $\mathbf{H}:=\bigoplus_{\kappa}\mathbb{Z}$, both ordered lexicographically. Basic set theory implies that $|G\times H|=\kappa$. If one endows the external direct sum $\mathbf{G}\oplus\mathbf{H}$ with the lexicographic order, then by Fact \ref{fact2}, there is a field $F$ of cardinality $\kappa$ and a valuation $v$ on $F$ with value group $\mathbf{G}\oplus\mathbf{H}$. Let $\pi_1\colon G\times H\rightarrow G$ be the projection function onto the first coordinate. It is straightforward to show that the map $w\colon F\rightarrow G\cup\{\infty\}$ defined by 

\begin{equation*}
w(x):=
\begin{cases}
(\pi_1\circ v)(x) & \text{if } x\neq 0, \ \textrm{and}\\
\infty & \text{otherwise }
\end{cases}
\end{equation*}

\noindent is a valuation on $F$ with value group $\mathbf{G}$. If $V$ is the corresponding valuation ring, then it follows from Fact \ref{fact3} and \eqref{eq3} above that $V$ has cardinality $\kappa$ and strong cardinal Krull dimension $\lambda$. \end{proof}

\subsection{The Upward-Restricted Version of Question \ref{qu1}}

We now change gears and consider the ``upward-restricted" version of Question \ref{qu1}. To begin, we recall the principal results from Bailey's thesis (\cite{AB}). 

\begin{theorem}[\cite{AB}, Theorems 5.7 and 6.3, respectively]\label{lem1} The following hold:

\begin{enumerate}
\item There are arbitrarily large cardinals $\kappa$ for which it can be shown in ZFC that there is a valuation ring $V$ of cardinality $\kappa$ with a chain of prime ideals of length $2^\kappa$.
\item Assuming GCH, for every infinite cardinal $\kappa$, there is a valuation ring $V$ of cardinality $\kappa$ with a chain of prime ideals of length $2^\kappa$.
\end{enumerate}
\end{theorem}

We combine Bailey's results with Theorem \ref{thm1} and Observation \ref{obs1} to immediately obtain

\begin{proposition}\label{prop3} The following is consistent with ZFC $($namely, it is a theorem of ZFC+GCH$)$: let $\kappa$ and $\lambda$ be cardinal numbers. Then there exists a ring $R$ of cardinality $\kappa$ and cardinal Krull dimension $\lambda$ if and only if
\begin{enumerate}
\item $\kappa\geq2$,
\item if $\kappa<\omega$, then $\lambda=0$, and
\item $\lambda\leq2^\kappa$. 
\end{enumerate}
If $($1$)$--$($3$)$ hold, we can choose such a ring $R$ with strong cardinal Krull dimension. Additionally, we can find such a valuation ring $R$ if and only if either $\kappa$ is an infinite cardinal or a power of a prime.\footnote{It is well-known that every finite local commutative ring has cardinality a power of a prime (see \cite{BM}). As valuation rings are local, the condition follows.} 
\end{proposition}

Our objective is now to study the limitations of Proposition \ref{prop3} in ZFC. First, we prove a proposition which will afford us a short proof of Bailey's two theorems presented above along with other results concerning what can be proved
in ZFC regarding the possible strong cardinal Krull dimension of a valuation ring.

\begin{proposition}\label{prop4'} Let $\kappa$ be an infinite cardinal and let $\lambda$ be the least cardinal such that $2^\lambda>\kappa$. Then there exists a totally ordered abelian group $\mathbf{G}$ such that $|G|=\kappa$ and $\mathbf{G}$ has at least $2^\lambda$ isolated subgroups. \end{proposition}

\begin{proof} We let $\kappa$ be an infinite cardinal and let $\lambda$ be the least cardinal such that $2^\lambda>\kappa$. Since  $\kappa$ is infinite, it is clear that 

\begin{equation}\label{eqsweet2}
\lambda\geq\aleph_0.
\end{equation}

\noindent Since $2^\kappa>\kappa$, we conclude that

\begin{equation}\label{eqthesweet}
\lambda\leq\kappa.
\end{equation}

Next, set  

\begin{equation}\label{eqbuddy}
2^{<\lambda}:=\{f\mid f\colon i\rightarrow\{0,1\} \ \textrm{is a function for some ordinal} \ i<\lambda\}.
\end{equation}

\noindent Observe that since $\lambda$ is a cardinal, for every ordinal $i<\lambda$, we have $|i|<\lambda$. By leastness of $\lambda$, it follows that $|\{f\mid f\colon i\rightarrow\{0,1\}\}|\leq\kappa$. It is immediate from this observation and \eqref{eqthesweet} that 

\begin{equation}\label{eqyet}
2^{<\lambda}\leq\kappa\cdot\kappa=\kappa.
\end{equation}

We now define the \emph{lexicographic order} $<$ on $2^{<\lambda+1}$ as follows: for $f,g\in2^{<\lambda+1}$, set $f<g$ if $f\subsetneq g$ or $f(i)<g(i)$, where $i$ is the least element of dom$(f)~\cap~$dom$(g)$. It is easy to see that $<$ is a total order on $2^{<\lambda+1}$. For an element $h\in2^{<\lambda+1}$, let $\mathrm{seg}(h):=\{x\in 2^{<\lambda+1}\mid x<h\}$. 

Next, let $\mathbf{G}:=\mathbb{Z}^{(2^{<\lambda})}$ denote the group (under coordinate-wise addition) of all functions $f\colon2^{<\lambda}\to\mathbb{Z}$ for which the \emph{support} $s(f):=f^{-1}(\mathbb{Z}\backslash\{0\})$ is finite. Elementary set theory gives

\begin{equation}\label{math}
|G|=\max(|\mathbb{Z}|,2^{<\lambda})=2^{<\lambda}\leq\kappa. 
\end{equation}

\noindent The \emph{reverse lexicographic order} on $G$ is defined as follows: for a nonzero $f\in G$, set $m(f):=\mathrm{max}(s(f))$. Now let $P:=\{f\in G\mid f(m(f))>0\}$. One checks easily that $P$ is closed under addition and that $(P,\{0\},-P)$ partitions $G$. Thus the order $\prec$ defined on $G$ by $f\prec g$ if and only if $g-f\in P$ induces a translation-invariant total order on the group $\mathbf{G}$.

For $x\in2^{<\lambda+1}$, we let $\mathbf{G}_x:=\{f\in G\mid s(f)\subseteq\mathrm{seg}(x)\}$. It is trivial to check that $\mathbf{G}_x$ is a subgroup of $\mathbf{G}$. We claim that 

\begin{equation}\label{drew}
\mathbf{G}_x \ \textrm{is an isolated subgroup of} \ \mathbf{G}. 
\end{equation}

\noindent To see this, suppose that $f\in G, g\in G_x$ and $0\prec f\prec g$. Suppose by way of contradiction that $f\notin G_x$. Then $s(f)\nsubseteq\mathrm{seg}(x)$. Hence there is $\varphi\in s(f)$ such that $x\leq\varphi$. But we now have $m(g)<x\leq m(f)$; thus

\begin{equation}\label{drew2}
m(g)<m(f). 
\end{equation}

\noindent Recall that $0\prec g$. Hence $g(m(g))>0$. We conclude that $(g-f)(m(g-f))=(g-f)(m(f))=g(m(f))-f(m(f))=-f(m(f))$. But recall that $0\prec f$, and thus $f(m(f))>0$. Finally, we see that $(g-f)(m(g-f))<0$. But this contradicts the fact that $0\prec g-f$, and \eqref{drew} is confirmed.

Following the notation presented in Section \ref{sub}, we let $\{0,1\}^\lambda\subseteq2^{<\lambda+1}$ denote the collection of all functions $f\colon\lambda\rightarrow\{0,1\}$. We claim that if $x,z$ are distinct members of $\{0,1\}^\lambda$, then $G_x\neq G_z$. Toward this end, it is easy to verify that $2^{<\lambda}$ is dense in $\{0,1\}^\lambda$. Without loss of generality, we may suppose that $x<z$. Then there exists $y\in2^{<\lambda}$ such that $x<y<z$. Consider the function $f\in G$ defined by 

\begin{equation*}
f(\zeta):=
\begin{cases}
1& \text{if } \zeta=y, \ \textrm{and}\\
0 & \text{otherwise }.
\end{cases}
\end{equation*}

\noindent Then it is clear that $f\in G_z\backslash G_x$. Therefore, the map $x\mapsto G_x$ is an injection from $\{0,1\}^\lambda$ into the collection of isolated subgroups of $\mathbf{G}$. Thus $\mathbf{G}$ has at least $2^\lambda>\kappa$-many isolated subgroups. Recall from \eqref{math} that $|G|\leq\kappa$. If $|G|=\kappa$, we have what we need. Otherwise, simply take any ordered group $\mathbf{H}$ with $H$ of cardinality $\kappa$ and order the external direct sum $\mathbf{G}\oplus\mathbf{H}$ lexicographically. One checks at once that for $x\in\{0,1\}^\lambda$, $\mathbf{G}_x\oplus\mathbf{H}$ is an isolated subgroup of $\mathbf{G}\oplus\mathbf{H}$, and the proof is complete.  \end{proof}

Our next corollary follows immediately from Fact \ref{fact2}, Fact \ref{fact3}, and the previous proposition.

\begin{corollary}\label{sun} For every infinite cardinal $\kappa$, there exists a valuation ring of cardinality $\kappa$ with strong cardinal Krull dimension strictly larger than $\kappa$. \end{corollary}

We are almost ready to establish Bailey's theorems as corollaries. First, we recall that an infinite cardinal $\kappa$ is a \emph{strong limit cardinal} provided that for every cardinal $\alpha<\kappa$, also $2^\alpha<\kappa$. We weaken this condition slightly and call an infinite cardinal $\kappa$ a \emph{pseudo strong limit cardinal} (PSL cardinal) if for every cardinal $\alpha<\kappa$, $2^\alpha\leq\kappa$. It is easy to construct arbitrarily large strong limits (hence arbitrarily large PSL cardinals). Begin with an infinite cardinal $\kappa_0$. By recursion, define $\kappa_{n+1}:=2^{\kappa_n}$. Now set $\kappa:=\bigcup_{n\geq 0}\kappa_n$. Then one checks at once that $\kappa$ is a strong limit. As promised, we give a short proof of Bailey's results.

\vspace{.1 in}

\emph{Proof of Theorem \ref{lem1}.} Let $\kappa$ be an infinite cardinal. By Corollary \ref{sun}, there exists a valuation ring $V$ of cardinality $\kappa$ and strong cardinal Krull dimension at least $\kappa^+$. So if we assume GCH, it follows that for every infinite cardinal $\kappa$, there is a valuation ring of cardinality $\kappa$ and strong cardinal Krull dimension $2^\kappa$. Now suppose that $\kappa$ is a PSL cardinal. Then it follows immediately from the definition of ``PSL cardinal" that the least cardinal $\lambda$ for which $2^\lambda>\kappa$ is $\kappa$ itself. By Fact \ref{fact2}, Fact \ref{fact3}, and Proposition \ref{prop4'}, we see that there is a valuation ring of cardinality $\kappa$ and strong cardinal Krull dimension $2^\kappa$. \hfill $\square$

\vspace{.1 in}

We now present an immediate corollary of our work above.

\begin{corollary}\label{condensed} Let $\kappa$ be an infinite cardinal. Then the following hold:
\begin{enumerate}
\item $($\cite{PK}, Problem 18.4$)$ There exists a set $X$ of cardinality $\kappa$ and a chain of $\kappa^+$ subsets of $X$.
\item Assume GCH. Then there is a set $X$ of cardinality $\kappa$ and a chain of $2^\kappa$ subsets of $X$. 
\item Suppose $\kappa$ is a PSL cardinal. Then there is a set $X$ of cardinality $\kappa$ and a chain of $2^\kappa$ subsets of $X$.
\end{enumerate}
\end{corollary}

Next, we prove that a conjecture of Bailey is undecidable in ZFC.

\begin{conjecture}[\cite{AB}, p. 33]\label{con1} The Generalized Continuum Hypothesis is true if and only if for every infinite cardinal number $\kappa$, there exists a commutative ring $R$ of cardinality $\kappa$ and cardinal Krull dimension strictly larger than $\kappa$. \end{conjecture}

\begin{proof} We have proved in ZFC that for every infinite cardinal $\kappa$, there is a valuation ring of cardinality $\kappa$ and strong cardinal Krull dimension strictly larger than $\kappa$. So ZFC + GCH proves the conjecture. On the other hand, ZFC+ $\neg$GCH disproves the conjecture. \end{proof}

Given our results up to this point, it is natural to consider the following question:

\begin{question}\label{qu2} Can one prove in ZFC that for every infinite cardinal $\kappa$, there exists a ring of cardinality $\kappa$ with cardinal Krull dimension $2^\kappa$?\end{question}

\noindent The answer is `no'. We now work toward proving that this is indeed the case, and begin with two lemmas.

\begin{lemma}\label{1000} Let $S$ be a set, and suppose that $\mathcal{C}$ is a chain of subsets of $S$. Define a binary relation $\prec$ on $S$ by letting $x\prec y$ whenever there exists $A\in\mathcal{C}$ which contains $x$ but not $y$. 

\begin{enumerate}
\item The relation $\prec$ is a $($strict$)$ partial order on $S$. 
\item There exists a subset $S'$ of $S$ such that

$($a$)$ the map $\mathcal{C}\to\mathcal{P}(S')$ defined by $A\mapsto A\cap S'$ is injective, and

$($b$)$ the restriction of $\prec$ to $S'$ is a $($strict$)$ total order on $S'$.
\end{enumerate}
\end{lemma}

\begin{proof} Let $S$, $\mathcal{C}$, and $\prec$ be as defined above.

\vspace{.1 in}

(1) It is patent that $\prec$ is irreflexive on $S$. Suppose now that $x,y,z\in S$ and $x\prec y\prec z$. Then there is $A_1\in\mathcal{C}$ which contains $x$ but not $y$ and $A_2\in\mathcal{C}$ which contains $y$ but not $z$. Because $\mathcal{C}$ is a chain, either $A_1\subseteq A_2$ or $A_2\subseteq A_1$. If $A_2\subseteq A_1$, then because $y\notin A_1$, also $y\notin A_2$, a contradiction. Hence $A_1\subseteq A_2$, and we deduce that $A_1$ contains $x$ but not $z$. By definition, $x\prec z$, and we have shown that $\prec$ is a partial order on $S$. 

\vspace{.1 in}

(2) Say that a set $X\subseteq S$ is \emph{$\mathcal{C}$-separated} if for every distinct $x,y\in X$, there exists $A\in\mathcal{C}$ such that either $A$ contains $x$ but not $y$ or $A$ contains $y$ but not $x$. By Zorn's Lemma, we may find a set $S'\subseteq S$ which is maximal with respect to being $\mathcal{C}$-separated. We claim that 

\begin{equation}\label{bee}
\textrm{the map} \ \mathcal{C}\to\mathcal{P}(S') \ \textrm{defined by} \ A\mapsto A\cap S' \ \textrm{is injective}. 
\end{equation}

\noindent For suppose by way of contradiction that $A_1, A_2$ are distinct members of $\mathcal{C}$ and $A_1\cap S'=A_2\cap S'$, yet $A_1\neq A_2$. Since $\mathcal{C}$ is a chain, we may suppose that $A_1\subsetneq A_2$. Choose $x^*\in A_2\backslash A_1$. Since $A_1\cap S'=A_2\cap S'$, we deduce that $x^*\notin S'$. Next, we claim that $S'\cup\{x^*\}$ is $\mathcal{C}$-separated. To see this, pick $x\in S'$ arbitrarily. It suffices to show that there is $A\in\mathcal{C}$ which contains $x$ but not $x^*$ or vice-versa. If $A_1$ contains $x$, then $A_1$ contains $x$ but not $x^*$. Otherwise, $x\notin A_1$. Because $A_1\cap S'=A_2\cap S'$ and $x\in S'$, we see that $x\notin A_2$. But then $A_2$ contains $x^*$ and not $x$. We have shown that $S'\cup\{x^*\}$ is $\mathcal{C}$-separated, and this contradicts the maximality of $S'$. We have now verified (2)(a). It is clear by the definition of $\prec$ and $S'$ that the restriction of $\prec$ to $S'$ is a total order, completing the proof of (2). \end{proof}

For a set $S$ and a chain $\mathcal{C}\subseteq\mathcal{P}(S)$, let us agree to call the partial order $\prec$ defined in Lemma \ref{1000} the \emph{$\mathcal{C}$-order} on $S$. Our next lemma is certainly known. Since we could not locate a reference, we include the proof.

\begin{lemma}\label{1001} Let $\kappa$ and $\lambda$ be cardinals such that $\kappa$ is infinite and $\lambda\geq\kappa$. Then there exists a linearly ordered set $\mathbf{B}:=(B,\leq)$ and a dense subset $\mathfrak{D}$ of $B$ such that $|B|=\lambda$ and $|\mathfrak{D}|=\kappa$ if and only if there exists a set $T$ of cardinality $\kappa$ and a chain $\mathcal{C}\subseteq\mathcal{P}(T)$ of cardinality $\lambda$. \end{lemma}

\begin{proof} Let $\kappa$ and $\lambda$ be as stated. Suppose first that there exists a linearly ordered set $\mathbf{S}:=(S,\leq)$ and a dense subset $\mathfrak{D}$ of $S$ such that $|S|=\lambda$ and $|\mathfrak{D}|=\kappa$. Now simply set $T:=\mathfrak{D}$. For $s\in S$, let $X_s:=\{x\in\mathfrak{D}\mid x\leq s\}$. Because $\mathfrak{D}$ is dense in $S$, we see that $X_{s_1}\neq X_{s_2}$ for $s_1\neq s_2$ in $S$. Setting $\mathcal{C}:=\{X_s\mid s\in S\}$ completes the proof of the forward implication. 

For the converse, suppose that $T$ is a set of cardinality $\kappa$ and that $\mathcal{C'}\subseteq\mathcal{P}(T)$ is a chain of cardinality $\lambda$. By Lemma \ref{1000}(2), there is a set $S\subseteq T$ and a chain $\mathcal{C}\subseteq\mathcal{P}(S)$ (namely, $\mathcal{C}:=\{X\cap S\mid X\in\mathcal{C'}\}$) such that $|\mathcal{C}|=\lambda$ and the $\mathcal{C}'$-order on $S$ is total. It is clear that this implies that the $\mathcal{C}$-order $\prec$ is total on $S$ as well. We now show that every $A\in\mathcal{C}$ is \emph{closed downward} relative to $\prec$; that is, we prove that 

\begin{equation}\label{eq9}
\textrm{if} \ a\in A, s\in S, \ \textrm{and} \ s\prec a, \ \textrm{then} \ s\in A.
\end{equation} 

\noindent This is easy to see: if $s\notin A$, then by definition of $\prec$, we have $a\prec s$, which is impossible since $s\prec a$ and $\prec$ is a total order on $S$. It is immediate from the linearity of $\prec$ on $S$ that

\begin{equation}\label{eq10}
\textrm{for any} \ s\in S, \ \mathrm{seg}(s):=\{x\in S\mid x\prec s\} \ \textrm{is closed downward relative to} \ \prec,\footnote{Note that this is true for \emph{any} linearly ordered set.}
\end{equation}

\noindent and

\begin{equation}\label{eq11}
\textrm{the subsets of} \ S \ \textrm{closed downward relative to} \ \prec \ \textrm{form a chain with respect to  inclusion.}
\end{equation}

We now extend $\mathbf{S}:=(S,\prec)$ to a dense linear order as follows: let $\mathbf{Q}:=(\mathbb{Q},<)$, where $<$ is the usual order on $\mathbb{Q}$. Next, order $S\times\mathbb{Q}$ lexicographically (relative to $\prec$ and $<$, respectively). For brevity, we denote the lexicographic order on $S\times\mathbb{Q}$ by $\vartriangleleft$. It is easy to see that 

\begin{equation}\label{fme}
\textrm{for every} \ X\in\mathcal{C}, X\times\mathbb{Q} \ \textrm{is closed downward relative to} \ \vartriangleleft \ \textrm{and has no} \ \vartriangleleft\textrm{-}\textrm{greatest element}. 
\end{equation}

\noindent As in the literature, call a set $Y\subseteq S\times\mathbb{Q}$ which has no greatest element and is closed downward with respect to $\vartriangleleft$ a \emph{cut}. Let $\mathfrak{D}':=\{\mathrm{seg}(z)\mid z\in S\times\mathbb{Q}\}$, and set $B:=\{X\times\mathbb{Q}\mid X\in\mathcal{C}\}\cup\mathfrak{D}'$. It is clear that $B$ has cardinality $\lambda$. One checks immediately that every member of $\mathfrak{D}'$ has no $\vartriangleleft$-greatest element, hence by \eqref{eq10}, every member of $\mathfrak{D}'$ is a cut. Moreover, $|\mathfrak{D}'|\leq\kappa$. We claim that $\mathfrak{D}'$ is dense in $B$ relative to $\subseteq$. To wit, suppose that $C_1$ and $C_2$ are distinct members of $B$. Without loss of generality, we may assume (by \eqref{eq11} and \eqref{fme}) that $C_1\subset C_2$ (that is, $C_1$ is a proper subset of $C_2$). Let $(a_0,q_0)\in C_2\backslash C_1$. As $C_2$ has no greatest element, there exist $(a_1,q_1), (a_2,q_2)\in C_2$ such that $(a_0,q_0)\vartriangleleft(a_1,q_1)\vartriangleleft(a_2,q_2)$. Then one verifies that $C_1\subset~\mathrm{seg}((a_1,q_1))\subset C_2$. Extend $\mathfrak{D}'$ to a subset $\mathfrak{D}$ of $B$ of cardinality $\kappa$ (if necessary). Then $\mathfrak{D}$ is clearly also dense in $B$. Finally, setting $\mathbf{B}:=(B,\subseteq)$, the proof is complete.  \end{proof}

Next, we recall a well-studied definition inspired by model theory.

\begin{definition}\label{defff} Let $\kappa$ be an infinite cardinal. Define $\mathrm{ded}(\kappa):=\mathrm{sup}\{\lambda\mid$ there is a linearly ordered set of cardinality $\lambda$ with a dense subset of cardinality $\kappa\}$. \end{definition}

Lemma \ref{1001} allows us to reformulate the definition of the $\mathrm{ded}$ function in a way more suitable to our purposes.

\begin{proposition}\label{1003} For any infinite cardinal $\kappa$, $\mathrm{ded}(\kappa)=\mathrm{sup}\{\lambda\mid$ there is a chain of subsets of $($a set of cardinality$)$ $\kappa$ of cardinality $\lambda\}$. \end{proposition}

\begin{proof} Immediate from Corollary \ref{condensed}(1) and Lemma \ref{1001}. \end{proof}

We now recall two results on the ded function from the literature which we shall employ shortly; we refer the reader to \cite{AC} and \cite{JK} for a fairly comprehensive sampling of what is known on this function. 

\begin{lemma}[\cite{AC}, \cite{WM}]\label{lastlem} Let $\kappa$ be an infinite cardinal.
\begin{enumerate}
\item $\kappa<\mathrm{ded}(\kappa)\leq2^\kappa$,\footnote{Observe that this is an immediate consequence of Corollary \ref{condensed}(1) and Proposition \ref{1003}.} and 
\item $($Mitchell$)$ it is consistent with ZFC that if $\mathrm{cf}(\kappa)>\aleph_0$, then $\mathrm{ded}(\kappa)<2^\kappa$.
\end{enumerate}
\end{lemma}

Recall from Observation \ref{obs1} that if $R$ is any ring, then $\mathrm{c.dim}(R)\leq2^{|R|}$. Applying Proposition \ref{1003} above yields a better bound.

\begin{proposition}\label{better bound} Let $R$ be an infinite ring of cardinality $\kappa$. Then $\mathrm{c.dim}(R)\leq\mathrm{ded}(\kappa)$. \end{proposition}

Finally, we close this section by proving that Question \ref{qu2} has a negative answer.

\begin{theorem}\label{yetanother} It is consistent with ZFC that if $\kappa$ is a cardinal of uncountable cofinality, then there exists a valuation ring $V$ of cardinality $\kappa$ such that $\kappa<\mathrm{sc.dim}(V)<2^\kappa$. \end{theorem}

\begin{proof} By Lemma \ref{lastlem}, it is consistent that $\kappa<\mathrm{ded}(\kappa)<2^\kappa$. Corollary \ref{sun} furnishes us with a valuation ring $V$ of cardinality $\kappa$ with strong cardinal Krull dimension strictly larger than $\kappa$. An application of Proposition \ref{better bound} concludes the proof. \end{proof}

\begin{corollary}\label{ttb} If $\kappa$ is a cardinal of uncountable cofinality, then it is undecidable in ZFC whether there exists a ring $R$ of cardinality $\kappa$ and cardinal Krull dimension $2^\kappa$. \end{corollary}

\begin{proof} By Lemma \ref{lastlem}, it is consistent that $\mathrm{ded}(\kappa)<2^\kappa$. In that case, if $R$ is any ring of cardinality $\kappa$, then $\mathrm{c.dim}(R)\leq\mathrm{ded}(\kappa)<2^\kappa$. Conversely, Theorem \ref{lem1}(2) yields the consistency of the existence of a valuation ring $V$ of cardinality $\kappa$ and strong cardinal Krull dimension $2^\kappa$. \end{proof}

\section{Polynomial Rings}

Employing some of the machinery extablished in the previous section, we now consider Question \ref{qu1} in the context of polynomial rings. We begin by establishing lower and upper bounds on the cardinal Krull dimension of polynomial rings and algebras over a division ring, respectively.

\begin{theorem}\label{anothernew} Suppose that $\kappa$ is an infinite cardinal.
\begin{enumerate}
\item If $R$ is any ring, then $\mathrm{ded}(\kappa)\leq\mathrm{c.dim}(R[X_i\mid i<\kappa])$. 
\item If $R$ is an algebra of dimension $\kappa$ $($as a left vector space$)$ over a division ring $D$, then $\mathrm{c.dim}(R)\leq\mathrm{ded}(\kappa)$. 
\end{enumerate}
\end{theorem}

\begin{proof} Let $\kappa$ be an infinite cardinal. 

\vspace{.1 in}

(1) Let $R$ be a ring and $\mathcal{V}$ be a chain of subsets of $\{X_i\mid i<\kappa\}$. Next, let $M$ be a maximal ideal of $R$. Then $R/M$ is a simple ring, hence prime. By Proposition 10.18 of \cite{LAM}, $(R/M)[X_i\mid i<\kappa]$ is also prime. But then for every $S\in\mathcal{V}$, $\langle S\rangle$ is a prime ideal of $(R/M)[X_i\mid i<\kappa]$ since $(R/M)[X_i\mid i<\kappa]/\langle S\rangle$ is isomorphic to a polynomial ring over $R/M$. It follows that there is a chain of prime ideals in $(R/M)[X_i\mid i<\kappa]$ of length $|\mathcal{V}|$. Recall from The Fundamental Theorem on Ring Homomorphisms that $(R/M)[X_i\mid i<\kappa]\cong R[X_i\mid i<\kappa]/M[X_i\mid i<\kappa]$. But then there is a chain of prime ideals of $R[X_i\mid i<\kappa]$ of length $|\mathcal{V}|$, each containing $M[X_i\mid i<\kappa]$. This establishes claim (1).

\vspace{.1 in}

\noindent (2) Suppose now that $R$ is an algebra of dimension $\kappa$ over a division ring $D$. Proceeding by contradiction, suppose that $\mathrm{ded}(\kappa)<\mathrm{c.dim}(R)$. Then there exists a cardinal $\lambda>\mathrm{ded}(\kappa)$ and a chain $\mathcal{C}$ of prime ideals of $R$ of length $\lambda$. By Lemma \ref{1000}(2), there is a set $S\subseteq R$ and a chain $\mathcal{C}'$ of subsets of $S$ such that the $\mathcal{C}$-order $\prec$ restricted to $S$ is total and such that $|\mathcal{C}'|=\lambda$.  

Our next claim is that 

\begin{equation}\label{bee3}
|S|>\kappa.
\end{equation}

\noindent To see this, suppose that $|S|\leq\kappa$. Recall above that there is a chain $\mathcal{C}'$ of subsets of $S$ of cardinality $\lambda>\mathrm{ded}(\kappa)$. Then certainly there exists a chain of subsets of $\kappa$ of cardinality $\lambda>\mathrm{ded}(\kappa)$, contradicting Proposition \ref{1003}.  

We now prove that 

\begin{equation}\label{bee4}
S \ \textrm{is linearly independent over} \ D. 
\end{equation}

\noindent Suppose not. Because $\prec$ is a total order on $S$, there exist distinct elements $s_0\prec s_1\cdots\prec s_n$ of $S$ and nonzero $d_0,\ldots,d_n\in D$ such that 

\begin{equation}\label{bee5}
0=d_0s_0+\cdots+d_{n}s_{n}. 
\end{equation}

\noindent By definition of $\prec$, for each $i$, $0\leq i< n$, we may choose $P_i\in\mathcal{C}$ such that $s_i\in P_i$ but $s_{i+1}\notin P_i$. Observe that $P_0\subsetneq P_1\cdots\subsetneq P_{n-1}$. Equation \eqref{bee5} and the fact that $D$ is a division ring yield that $s_n\in P_{n-1}$, a contradiction.

We have shown that there exists a linearly independent subset $S$ of $R$ of size greater than the dimension of $R$ as a $D$-algebra. This is absurd, and the proof is complete.  \end{proof}

Next, we extend Proposition \ref{prop1} from the Introduction.

\begin{corollary}\label{thm2} Let $\kappa$ be an infinite cardinal and let $R$ be a ring. 
\begin{enumerate}
\item If $|R|\leq\kappa$, then $\mathrm{c.dim}(R[X_i\mid i<\kappa])=\mathrm{ded}(\kappa)$. 
\item If $R$ is a division ring, then $\mathrm{c.dim}(R[X_i\mid i<\kappa])=\mathrm{ded}(\kappa)$. 
\item If $|R|\leq\kappa$ and $\kappa$ is PSL, then $\mathrm{sc.dim}(R[X_i\mid i<\kappa])=2^\kappa.$ 
\item If $|R|\leq\kappa$ and GCH holds, then $\mathrm{sc.dim}(R[X_i\mid i<\kappa])=2^\kappa$.
\end{enumerate}
\end{corollary}

\begin{proof} Assume $\kappa$ is an infinite cardinal and $R$ is a ring. 

\vspace{.1 in}

(1) If $|R|\leq\kappa$, then the ring $R[X_i\mid i<\kappa]$ has cardinality $\kappa$. Proposition \ref{better bound} and Theorem \ref{anothernew}(1) yield the desired result. 

\vspace{.1 in}

(2) Immediate from Theorem \ref{anothernew}.

\vspace{.1 in}

(3) Assume that $\kappa$ is PSL. By Corollary \ref{condensed}(3), there is a chain $\mathcal{C}$ of $2^\kappa$ subsets of $\kappa$. By the proof of Theorem \ref{anothernew}(1), we obtain a chain of $2^\kappa$ prime ideals of $R[X_i\mid i<\kappa]$. 

\vspace{.1 in}

(4) Under GCH, every infinite cardinal is PSL. \end{proof}

We now present an analog of Theorem \ref{yetanother} for polynomial rings. 

\begin{corollary}\label{cccc} It is consistent with ZFC that if $\kappa$ is a cardinal of uncountable cofinality and $R$ is a ring of size at most $\kappa$, then $\kappa<\mathrm{c.dim}(R[X_i\mid i<\kappa])<2^\kappa$. \end{corollary}

\begin{proof} Immediate from Corollary \ref{thm2}(1) and Lemma \ref{lastlem}. \end{proof}

The next two facts from the literature will enable us to prove another corollary along with a proposition.

\begin{fact}\label{fact4}[\cite{RG}, p. 106] Let $R$ and $S$ be commutative rings with $S$ integral over $R$. Then either $R$ and $S$ both have infinite cardinal Krull dimension or $\mathrm{sc.dim}(R)=\mathrm{sc.dim}(S)$. \end{fact}

\begin{fact}\label{fact5}[\cite{GO}, Corollary 4] Let $R$ be an infinite ring satisfying ACC on (two-sided) ideals. Then there exists a prime ideal $P$ of $R$ such that $|R/P|=|R|$. \end{fact}

\begin{corollary} \label{cor1} The following hold:
\begin{enumerate} 
\item Let $S$ be a commutative domain or a commutative Noetherian ring of PSL cardinality $\kappa>\aleph_0$. Then there exists an extension of rings $R\subseteq S\subseteq T$ such that $|R|=|T|=\kappa$ and $\mathrm{sc.dim}(R)=\mathrm{sc.dim}(T)=2^\kappa$. 
\item Assume GCH. Then for every cardinal $\alpha>\aleph_0$: if $S$ is a commutative domain or a commutative Noetherian ring of cardinality $\alpha$, then there exist $R$ and $T$ as in $($1$)$.
\end{enumerate}
\end{corollary}

\begin{proof} (1) Let $\kappa>\aleph_0$ be PSL. Suppose first that $S$ is a commutative domain of  cardinality $\kappa$, and let $P(S)$ be the prime subring of $S$. Note that $P(S)$ is countable. Basic set theory yields a subset $\mathcal{X}$ of $S$ of cardinality $\kappa$ which is algebraically independent over $P(S)$. Hence to within isomorphism, $P(S)[X_i\mid i<\kappa]$ is a subring of $S$. Making this identification, we have $P(S)[X_i\mid i<\kappa]\subseteq S\subseteq S[X_i\mid i<\kappa]$. We now invoke (3) of Corollary \ref{thm2} to conclude the existence of $R$ and $T$ with the desired properties. Next, suppose that $S$ is Noetherian of cardinality $\kappa$. By Fact \ref{fact5}, there is a prime ideal $P$ of $S$ such that $|S/P|=|S|$. By our work above, there is a subring of $S/P$ of cardinality $\kappa$ and strong cardinal Krull dimension $2^\kappa$. This subring is of the form $R/P$ for some subring $R$ of $S$ containing $P$. But then $R$ has a chain of $2^\kappa$ prime ideals, each containing $P$. It follows that $\mathrm{sc.dim}(R)=2^\kappa$. As before, we set $T:=S[X_i\mid i<\kappa]$. 

(2) Immediate from (1) since under GCH, every infinite cardinal is PSL.\end{proof}

\begin{remark} A natural question is whether we can drop the assumption that $S$ is either a domain or Noetherian in Corollary \ref{cor1}. The answer is no. To see this, let $\kappa$ be an arbitrary infinite cardinal, and set $R:=\mathbb{Z}[X_i\mid i<\kappa]/\langle X_iX_j\mid i,j<\kappa\rangle$. Observe that the variables embed injectively into $R$ in the canonical way, and hence $R$ has cardinality $\kappa$. Moreover, every element $r\in R$ can be expressed in the form $r=m+x$, where $m$ is an integer and $x^2=0$ (modulo $\langle X_iX_j\mid i,j<\kappa\rangle$). It follows that $R$ is an integral extension of $\mathbb{Z}$ (we identify $\mathbb{Z}$ with its image in $R$), and hence every subring of $R$ has strong Krull dimension $1$, by Fact \ref{fact4}. \end{remark}

It is easy to see via an elementary set-theoretic argument that if $F$ is an infinite field and $K$ is an algebraic field extension of $F$, then $F$ and $K$ have the same cardinality. Employing a technique used in the proof of Corollary \ref{cor1}, we generalize this result.

\begin{proposition}\label{prop5} Let $R$ be an infinite commutative ring, and suppose that $S$ is a commutative Noetherian integral extension of $R$. Then $|R|=|S|$. \end{proposition}

\begin{proof} Let $R$ and $S$ be as postulated. Suppose by way of contradiction that $|S|>|R|$. Let $P$ be a prime ideal of $S$ such that $|S|=|S/P|$ (furnished by Fact \ref{fact5}). Then $S/P$ is an integral extension of $R/(R\cap P)$ (\cite{RG}, Lemma 10.1). Observe that $S/P$ and $R/(R\cap P)$ are domains and that $|S/P|=|S|>|R|\geq|R/(R\cap P)|$. Let $K$ and $F$ be the quotient fields of $S/P$ and $R/(R\cap P)$, respectively. Then $K$ is an algebraic field extension of $F$, yet $|K|>|F|$ (note that the cardinality of any commutative domain is the same as the cardinality of its quotient field). From the remarks preceding the statement of Proposition \ref{prop5}, we deduce that $F$ is finite. Because $R$ is infinite and $|S|>|R|$, $S$ is uncountable. But then $K$ is an uncountable algebraic field extension of the finite field $F$, and we have reached a contradiction. \end{proof}

We close this section by showing that it is consistent with ZFC that, in a sense, Conjecture \ref{con1} can fail quite badly. As a full proof of the set-theoretic assertions would take us too far afield, we give an outline of the argument and refer the reader to the popular texts \cite{TJ} and \cite{KK} for references on forcing. We also remark that neither the next proposition nor the technique of forcing will be utilized in the remainder of this note.

\begin{proposition}\label{forceprop} There is a model of ZFC where CH fails, yet for every infinite cardinal $\kappa$, there is a commutative ring $R$ of cardinality $\kappa$ and strong cardinal Krull dimension $2^\kappa$. \end{proposition}

\begin{proof} [Sketch of Proof] Begin with a countable transitive model of ZFC + GCH, and add $\aleph_2$ many Cohen reals. A straightforward (Cohen) forcing argument shows that GCH fails only at $\omega$ (this is an immediate consequence of \cite{KK}, exercise G1, for example). In other words, we obtain a model of ZFC which satisfies
\begin{enumerate}
\item $2^{\aleph_0}=\aleph_2$, and
\item for all cardinals $\beta\geq\aleph_1$, $2^\beta=\beta^+$ (the least cardinal larger than $\beta$). 
\end{enumerate}

\noindent We claim that every infinite cardinal $\beta\geq\aleph_2$ is PSL in this model. It follows immediately from (1) and (2) above that $\aleph_2$ is PSL. Suppose that $\beta>\aleph_2$, and assume that $\alpha<\beta$. We must show that $2^\alpha\leq\beta$. If $\alpha\geq\aleph_1$, then by (2), $2^\alpha=\alpha^+\leq\beta$. Otherwise, $\alpha \leq \aleph_0$. In that case, (1) yields $2^\alpha \leq \aleph_2<\beta$. Therefore $\beta$ is PSL, and hence Corollary \ref{thm2}(3) implies that for every cardinal $\beta\geq\aleph_2$, there is a commutative ring of cardinality $\beta$ and strong cardinal Krull dimension $2^\beta$. Suppose now that $\beta=\aleph_0$ or $\beta=\aleph_1$, and let $F$ be a field of cardinality $\beta$. Observe that the polynomial ring $F[X_i\mid i<\omega]$ has cardinality $\beta$. Moreover, the proof of Proposition \ref{prop1} along with (1) and (2) above shows that the strong cardinal Krull dimension of $F[X_i\mid i<\omega]$ exists and is equal to $2^{\aleph_0}=\aleph_2=2^\beta$. The proof is now concluded. \end{proof}

\section{Leavitt Path Algebras}

With the help of Leavitt path algebras, we give another construction of a class of (mostly noncommutative) rings whose strong cardinal Krull dimensions exceed their cardinalities. For the reader who is not familiar with Leavitt path algebras, we begin with a very brief introduction to these rings. 

\vspace{.1 in}

\begin{center}\emph{Throughout this section, rings will be assumed only to be associative, not necessarily unital.}\end{center}

\subsection{More on Partially Ordered Sets} 

Suppose $\mathbf{P}:=(P, \leq)$ is a partially ordered set. We call $\mathbf{P}$ \emph{downward directed} if $P$ is nonempty, and for all $p,q \in P$ there exists $r \in P$ such that $p \geq r$ and $q\geq r$. An element $x\in P$ is called a \emph{least}, respectively \emph{greatest}, element of $P$ if $x \leq y$, respectively $y \leq x$, for all $y \in P$. Let $(P_1, \leq_1)$ and $(P_2, \leq_2)$ be two partially ordered sets, and let $f:P_1 \to P_2$ be a function. Then $f$ is \emph{order-preserving} if $x\leq_1 y$ implies that $f(x) \leq_2 f(y)$ for all $x,y \in P_1$. Also, $f$ is \emph{order-reflecting} if $f(x) \leq_2 f(y)$ implies that $x \leq_1 y$ for all $x, y \in P_1$. In the latter case, $f$ is necessarily injective, since $f(x)=f(y)$ implies that $x \leq_1 y$ and $y \leq_1 x$. If $f$ is both order-preserving and order-reflecting, then it is an \emph{order-embedding}. If $f$ is an order-embedding and bijective, then $f$ is an \emph{order-isomorphism}. We write $(P_1, \leq_1) \cong (P_2, \leq_2)$ if there is an order-isomorphism between the two partially ordered sets.

\subsection{The Leavitt Path Algebra Construction}

A \emph{directed graph} $\mathbf{E}:=(E^0,E^1,\so, \ra)$ consists of two disjoint sets $E^0,E^1$ (the elements of which are called \emph{vertices} and \emph{edges}, respectively), together with functions $\so,\ra\colon E^1 \to E^0$, called \emph{source} and \emph{range}, respectively. We shall refer to directed graphs as simply ``graphs" from now on. A \emph{path} $p$ of \emph{length} $n$ in $\mathbf{E}$ is a sequence $e_1\cdots e_n$ of edges $e_1,\dots, e_n \in E^1$ such that $\ra(e_i)=\so(e_{i+1})$ for $i\in \{1,\dots,n-1\}$. We view the elements of $E^0$ as paths of length $0$, and denote by $\pth(\mathbf{E})$ the set of all paths in $\mathbf{E}$. Given a vertex $v \in E^0$, if the set $\so^{-1}(v) = \{e\in E^1\mid \so(e)=v\}$ is finite but nonempty, then $v$ is said to be a \emph{regular vertex}. 

Next, let $K$ be a field and $\mathbf{E}$ be a nonempty graph\footnote{That is, $E^0$ is nonempty.}. Augment $E^1$ by adjoining a new edge $e^*$ for every $e\in E_1$ such that

\begin{enumerate}
\item for $e_1\neq e_2$, $e_1^*\neq e_2^*$, and 
\item $(E^1)^*:=\{e^*\mid e\in E_1\}$ is disjoint from $E^1$ and $E^0$.
\end{enumerate}

\noindent Extend $\so$ and $\ra$ to $E^1\cup(E^1)^*$ by setting $\so(e^*):=\ra(e)$ and $\ra(e^*):=\so(e)$. The \emph{Leavitt path algebra} $L_K(\mathbf{E})$ \emph{of $\mathbf{E}$} is the $K$-algebra freely generated over $K$ by $E^0\cup E^1\cup(E^1)^*$, modulo the following relations (where $\delta$ denotes the Kronecker delta):

\smallskip

{(V)} \ \ \ \  $vw = \delta_{v,w}v$ for all $v,w\in E^0$,

{(E1)} \ \ \ $\so(e)e=e\ra(e)=e$ for all $e\in E^1$,

{(E2)} \ \ \ $\ra(e)e^*=e^*\so(e)=e^*$ for all $e\in E^1$,

{(CK1)}  \hspace{.01in}  $e^*f=\delta _{e,f}\ra(e)$ for all $e,f\in E^1$, and

{(CK2)} \  $v=\sum_{e\in \so^{-1}(v)} ee^*$ for every regular vertex $v\in E^0$.   

\smallskip

For all $v \in E^0$ we define $v^*:=v$, and for all paths $p:=e_1 \cdots e_n$ ($e_1, \dots, e_n \in E^1$) we set $p^*:=e_n^* \cdots e_1^*$. With this notation, every element of $L_K(\mathbf{E})$ can be expressed (though not necessarily uniquely) in the form $\sum_{i=1}^n k_ip_iq_i^*$ for some $k_i \in K$ and $p_i,q_i \in \pth (\mathbf{E})$. It is also easy to establish that $L_K(\mathbf{E})$ is unital (with multiplicative identity $\sum_{v\in E^0}v$) if and only if $E^0$ is a  finite set.    

For additional information and background on Leavitt path algebras see, e.g., \cite{ASurvey}.  

\subsection{Associating Leavitt Path Algebras to Posets}

We begin with several definitions and a theorem from~\cite{AAMS}, which we record next.

\begin{definition} \label{mainconstr}
Given a partially ordered set $\mathbf{P}:=(P, \leq)$ we define the graph $\mathbf{E}_\mathbf{P}$ as follows:
$$E^0_\mathbf{P}:=\{v_p \mid p \in P\} \ \ \ \ \text{ and }  \ \ \ \ \ E^1_\mathbf{P}:=\{e_{p,q}^i \mid i \in \N,  \text{ and } p,q \in P \text{ satisfy } p>q\},$$ where $\so(e_{p,q}^i):=v_p$ and $\ra(e_{p,q}^i):=v_q$ for all $i \in \N$.
\end{definition}

Intuitively, we build $\mathbf{E}_\mathbf{P}$ from $\mathbf{P}$ by viewing the elements of $P$ as vertices and putting countably-infinitely many edges between them whenever the corresponding elements are related in $\mathbf{P}$. It turns out that the prime spectrum of a Leavitt path algebra constructed from a graph of this form has a structure closely related to that of $\mathbf{P}$. The next two definitions will enable us to state this fact precisely.

\begin{definition} \label{precdef}
Let $(P, \leq)$ be a partially ordered set. Define a binary relation $\preceq$ on the collection of nonempty subsets of $P$ as follows. Given $S_1, S_2 \in \mathcal{P}(P) \setminus \{\varnothing\}$, write $S_1 \preceq S_2$ if for every $s_2 \in S_2$ there exists $s_1 \in S_1$ such that $s_1 \leq s_2$. If $S_1 \preceq S_2$ and $S_2 \preceq S_1$, then we write $S_1 \approx S_2$. 
\end{definition}

It is routine to verify that $(\mathcal{P}(P) \setminus\{\varnothing\}, \preceq)$ is a preordered set,  that $\approx$ is an equivalence relation on $\mathcal{P}(P) \setminus\{\varnothing\}$, and that  $\preceq$ induces a partial order on the set $(\mathcal{P}(P) \setminus\{\varnothing\})/\approx$ of $\approx$-equivalence classes. For $S \in \mathcal{P}(P) \setminus\{\varnothing\}$, we shall denote the $\approx$-equivalence class of $S$ by $[S]$.

\begin{definition} \label{leqAdef}
Given a partially ordered set $(P, \leq)$, let $$\AT(P):=P \cup \{x_{[S]} \mid S \subseteq P \text{ is downward directed with no least element}\}.$$ Further, we extend $\leq$ to a binary relation $\leq_\AT$ on $\AT(P)$, as follows.  For all $p,q \in \AT(P)$ let $p \leq_\AT q$ if one of the following holds:
\begin{enumerate}
\item[$(1)$] $p,q \in P$ and $p \leq q$,
\item[$(2)$] $p \in P$, $q = x_{[S]} \in \AT(P)\setminus P$, and $p \leq s$ for all $s \in S$,
\item[$(3)$] $p = x_{[S]} \in \AT(P)\setminus P$, $q \in P$, and $s \leq q$ for some $s \in S$, or
\item[$(4)$] $p,q \in \AT(P)\setminus P$, $p = x_{[S]}$, $q = x_{[T]}$, and $S \preceq T$. 
\end{enumerate}
\end{definition}

We conclude this subsection by recalling a useful result from \cite{AAMS} and proving a lemma.

\begin{theorem}[\cite{AAMS}, Theorem 18] \label{EPspec}
Let $(P, \leq)$ be a partially ordered set and let $K$ be a field. Then $(\AT(P),  \leq_\AT)$ is a partially ordered set, and $$(\Spec(L_K(\mathbf{E}_\mathbf{P})), \subseteq) \ \cong \  (\AT(P), \leq_\AT),$$ where $\Spec(L_K(\mathbf{E}_\mathbf{P}))$ denotes the set of prime ideals of $L_K(\mathbf{E}_\mathbf{P})$.
\end{theorem}

\begin{lemma} \label{chain-lemma}
If $(P, \leq)$ is linearly ordered, then so is $(\AT(P), \leq_\AT)$.
\end{lemma}

\begin{proof}
By Theorem~\ref{EPspec}, $(\AT(P), \leq_\AT)$ is a poset. Now assuming that $(P, \leq)$ is linearly ordered, it suffices to take two arbitrary elements $p, q \in \AT(P)$ and prove that either $p \leq_\AT q$ or $q \leq_\AT p$. In light of Definition~\ref{leqAdef}, there are four cases to consider. 

First, suppose that $p,q \in P$. Then either $p \leq q$ or $q \leq p$, since $P$ is a chain. Hence either $p \leq_\AT q$ or $q \leq_\AT p$, by Definition~\ref{leqAdef}(1).

Second, suppose that $p \in P$ and $q = x_{[S]} \in \AT(P)\setminus P$, for some $S \subseteq P$. Since $P$ is linearly ordered, either $p \leq s$ for all $s \in S$, or $t \leq p$ for some $t \in S$. In the first case, $p \leq_\AT q$, by Definition~\ref{leqAdef}(2), while in the second case, $q \leq_\AT p$, by Definition~\ref{leqAdef}(3).

The case where $q \in P$ and $p \in \AT(P)\setminus P$ is analogous to the previous one.

Finally, suppose that $p,q \in \AT(P)\setminus P$, $p = x_{[S]}$, and $q = x_{[T]}$, for some $S, T \subseteq P$. Since $P$ is linearly ordered, we must have either $S \preceq T$ or $T \preceq S$. In the first case, $p \leq_\AT q$, while in the second case, $q \leq_\AT p$, by Definition~\ref{leqAdef}(4).
\end{proof}

\subsection{Cardinal Krull Dimension of Leavitt Path Algebras}

Our next goal is to compute the cardinality and strong cardinal Krull dimension of a Leavitt path algebra of the form $L_K(E_\mathbf{P})$.

\begin{lemma} \label{card-lemma}
Let $K$ be a field and $\mathbf{E}$ a graph. Also set $|K|:= \kappa$ and $|\mathbf{E}|:=|E^0 \cup E^1|:= \epsilon$. Then the following hold:
\begin{enumerate}
\item[$(1)$] $|L_K(\mathbf{E})| \leq \aleph_0 \cdot \kappa \cdot \epsilon^2$.
\item[$(2)$] If $\kappa \leq \epsilon$ and $\, \aleph_0 \leq \epsilon$, then $|L_K(\mathbf{E})| = \epsilon$.
\end{enumerate}
\end{lemma}

\begin{proof}
(1) First, note that $$|\pth (\mathbf{E})| \leq |E^0| + \sum_{n=1}^\infty |E^1|^n \leq |E^0| + \aleph_0\cdot |E^1| \leq \aleph_0 \cdot (|E^0| + |E^1|) = \aleph_0 \cdot \epsilon.$$ As mentioned above, every element of $L_K(\mathbf{E})$ can be expressed in the form $\sum_{i=1}^n k_ip_iq_i^*$ for some $k_i \in K$ and $p_i,q_i \in \pth (\mathbf{E})$. Hence $|L_K(\mathbf{E})| \leq \aleph_0 \cdot \kappa \cdot \epsilon^2$.

(2) Since $E \subseteq L_K(\mathbf{E})$, necessarily $\epsilon \leq |L_K(\mathbf{E})|$. If $\kappa \leq \epsilon$ and $\aleph_0 \leq \epsilon$, then $\aleph_0 \cdot \kappa \cdot \epsilon^2 = \epsilon,$ and hence $|L_K(\mathbf{E})| = \epsilon$, by (1).
\end{proof}

\begin{corollary} \label{card-cor}
Let $K$ be a field and $(P, \leq)$ a poset. If $|K| \leq |P|$ and $\aleph_0 \leq |P|$, then $|L_K(\mathbf{E}_\mathbf{P})| =  |P|$.
\end{corollary}

\begin{proof}
We observe that $$|\mathbf{E}_\mathbf{P}| = |E_\mathbf{P}^0 \cup E_\mathbf{P}^1| = |P| + |E_\mathbf{P}^1| \leq |P| + \aleph_0 \cdot |P|^2.$$ Hence if $\aleph_0 \leq |P|$, then $|\mathbf{E}_\mathbf{P}| = |P|$. Therefore $|L_K(\mathbf{E}_\mathbf{P})| =  |P|$, by Lemma~\ref{card-lemma}(2).
\end{proof}

\begin{proposition} \label{A-Krull-prop}
Let $K$ be a countable field, $(P, \leq)$ an infinite linearly ordered structure, and $R:= L_K(\mathbf{E}_\mathbf{P})$. Then $|R| = |P|$ and $\mathrm{sc.dim}(R) = |\AT(P)|$.
\end{proposition}

\begin{proof}
The equality $|R| = |P|$ follows from Corollary~\ref{card-cor}. Since $(P, \leq)$ is linearly ordered, so is $(\AT(P), \leq_\AT)$, by Lemma~\ref{chain-lemma}. Hence $\mathrm{sc.dim}(R) = |\AT(P)|$, by Theorem~\ref{EPspec}.
\end{proof}

%Given a linearly ordered structure $ (B, \leq)$, we say that $P \subseteq B$ is \emph{dense} in $B$ if for all $r,p \in B$ satisfying $r<p$ there exists $q \in P$ such that $r < q < p$.

\begin{corollary} \label{dense-cor}
Let $K$ be a countable field and $\kappa$ an infinite cardinal. Suppose there exists a linearly ordered structure $(B, \leq)$ and a dense set $\mathfrak{D} \subseteq B$ such that $|B|=\lambda$, $|\mathfrak{D}|=\kappa$, and $\lambda \geq \kappa$. Setting $R = L_K(\mathbf{E}_\mathfrak{D})$, we have $|R| = \kappa$ and $\mathrm{sc.dim}(R) \geq \lambda$.
\end{corollary}

\begin{proof}
By Proposition~\ref{A-Krull-prop}, it is enough to show that $|\AT(\mathfrak{D})| \geq \lambda$. To prove this inequality, we shall construct an injective function $\phi : B\setminus \mathfrak{D} \to \AT(\mathfrak{D}) \setminus \mathfrak{D}$. Without loss of generality, we may assume that if $B$ has a greatest element, then it belongs to $\mathfrak{D}$.

Given $r \in B \setminus \mathfrak{D}$, let $S_r:=\{p \in \mathfrak{D} \mid r < p\}$. Since $r$ is not a greatest element, by hypothesis, $S_r \neq \varnothing$. Thus $S_r$ is downward directed (a chain, actually, since it is a subset of a chain) and has no least element, since $\mathfrak{D}$ is dense in $B$. Therefore $x_{[S_r]} \in \AT(\mathfrak{D}) \setminus \mathfrak{D}$, and hence $r\mapsto x_{[S_r]}$ defines a function $\phi : B\setminus \mathfrak{D} \to \AT(\mathfrak{D}) \setminus \mathfrak{D}$. To show that $\phi$ is injective, let $r,p \in B \setminus \mathfrak{D}$ be distinct. Without loss of generality, we may assume that $r < p$. Then there exists $q \in \mathfrak{D}$ such that $r < q < p$, since $\mathfrak{D}$ is dense in $B$. Therefore $q \in S_r \setminus S_p$, which implies that $S_r \prec S_p$. Thus $x_{[S_r]} <_\AT x_{[S_p]}$, and in particular, $x_{[S_r]} \neq x_{[S_p]}$.
\end{proof}

Our next corollary is a Leavitt path algebra analogue of Proposition \ref{prop1}(2).

\begin{corollary}
Let $K$ be a countable field and $R:=L_K(\mathbf{E}_{\Q})$, where $\Q$ is the set of the rational numbers with the usual ordering $\leq$. Then $|R| = \aleph_0$ and $\mathrm{sc.dim}(R) = 2^{\aleph_0}$.
\end{corollary}

\begin{proof}
Setting $B:=\R$ and $P:=\Q$, Corollary~\ref{dense-cor} implies that $|R| = \aleph_0$ and $\mathrm{sc.dim}(R) \geq 2^{\aleph_0}$. The desired conclusion now follows from the fact that, by Proposition~\ref{A-Krull-prop}, $|\mathrm{sc.dim}(R)| = |\AT(\Q)| \leq |\Q| + 2^{|\Q|} = 2^{\aleph_0}$. 
\end{proof}

Combining Corollary~\ref{dense-cor} with Lemma~\ref{1001}, we obtain our final theorem.

\begin{theorem}\label{thm7}
Let $K$ be a countable field and $\kappa$ an infinite cardinal. Suppose there exists a set of cardinality $\kappa$ whose power set contains a linearly ordered subset of cardinality $\lambda \geq \kappa$. Then there exists a graph $\mathbf{E}$ such that $R:=L_K(\mathbf{E})$ satisfies $|R| = \kappa$ and $\mathrm{sc.dim}(R) \geq \lambda$.
\end{theorem}

An application of Corollary \ref{condensed} and Observation \ref{obs1} yields the following corollary:

\begin{corollary}
Let $K$ be a countable field and $\kappa$ an infinite cardinal. 
\begin{enumerate}
\item There exists a graph $\mathbf{E}$ such that $R:=L_K(\mathbf{E})$ satisfies $|R| = \kappa$ and $\mathrm{sc.dim}(R) \geq \kappa^+$.
\item Assuming GCH, there exists a graph $\mathbf{E}$ such that $R:=L_K(\mathbf{E})$ satisfies $|R| = \kappa$ and $\mathrm{sc.dim}(R) = 2^\kappa$.
\item If $\kappa$ is a PSL cardinal, then there exists a graph $\mathbf{E}$ such that $R:=L_K(\mathbf{E})$ satisfies $|R| = \kappa$ and $\mathrm{sc.dim}(R) = 2^\kappa$.
\end{enumerate}
\end{corollary}

Finally, we use Leavitt path algebras to construct rings with cardinal Krull dimension of every limit cardinality but without strong cardinal Krull dimension.

\begin{proposition}\label{berry1}
Let $K$ be a field and $\kappa$ a nonzero limit cardinal. Then there exists a graph $\mathbf{E}$ such that $R:=L_K(\mathbf{E})$ has cardinal Krull dimension $\kappa$ but does not have strong cardinal Krull dimension.
\end{proposition}

\begin{proof}
For each cardinal $\lambda < \kappa$ let $(P_\lambda, \leq_\lambda)$ be a well-ordered set such that $|P_\lambda| = \lambda$. Also let $P:=\bigsqcup_{\lambda < \kappa} P_\lambda$ (the disjoint union of the $P_\lambda$), and define a partial order $\leq$ on $P$ by letting $s \leq t$, for all $s,t \in P$, whenever $s,t \in P_\lambda$ and $s \leq_\lambda t$ for some $\lambda < \kappa$. Then every downward directed subset of $P$ is a subset of some $P_\lambda$, and hence has a least element, since $\leq_\lambda$ is a well-order. It follows that $(P, \leq) =  (\AT(P), \leq_\AT)$, and hence $(\Spec(L_K(\mathbf{E}_\mathbf{P})), \subseteq) \cong (P, \leq)$, by Theorem~\ref{EPspec}. Therefore, for each cardinal $\lambda < \kappa$, $R = L_K(\mathbf{E}_\mathbf{P})$ has a chain of prime ideals of cardinality $\lambda$ (corresponding to $P_\lambda$), but no chains of prime ideals of cardinality $\kappa$. Hence $\mathrm{c.dim} (R) = \mathrm{sup} \{\lambda \mid \lambda < \kappa\} = \kappa$, but $R$ does not have strong cardinal Krull dimension. \end{proof}

\begin{remark} Given an infinite field $F$ and a graph $\mathbf{E}$, it can be shown that the ideal extension $E(F, L_K(\mathbf{E}))$ of $F$ by $L_K(\mathbf{E})$ is a unital ring, having the same cardinality as $L_K(\mathbf{E})$, such that there is a one-to-one inclusion-preserving correspondence between the prime ideals of $L_K(\mathbf{E})$ and the prime ideals of $E(F, L_K(\mathbf{E}))$, except that $E(F,L_K(\mathbf{E}))$ has an additional maximal ideal. Thus the conclusions of the last three results hold with $E(F, L_K(\mathbf{E}))$ in place of $L_K(\mathbf{E})$. \end{remark}

\begin{center}\noindent\textbf{\textit{Acknowledgments}}\end{center}

The authors warmly thank many mathematicians who contributed to the results presented in this paper. First, we are grateful to John Griesmer and Noah Schweber for pointing us to the reference \cite{KK} which played a pivotal role in the proof of Proposition \ref{forceprop}. We also acknowledge Joel David Hamkins' comments on an older question posted on \texttt{http://www.mathoverflow.net} which were of utility to us in constructing the proof of Lemma \ref{1001}. Be'eri Greenfeld's feedback on an older version of this manuscript inspired the proofs of Theorem \ref{anothernew}(2) and Proposition \ref{berry1}. We thank ``n\'ombre" for ideas which were a big help to us in constructing the proof of Proposition \ref{prop4'}. Finally, we are grateful to the referee for pointing us to relevant literature and for comments which improved the readability of this article.

\bibliographystyle{plain}

\begin{thebibliography}{10}

\bibitem{ASurvey}  
G.~Abrams, 
{\it Leavitt path algebras: the first decade},  
Bull. Math. Sci. \textbf{5}(1) (2015) 59--120. 

\bibitem{AAMS} 
G.~Abrams, G.~Aranda Pino, Z.~Mesyan, C.~Smith, 
{\it Realizing posets as prime spectra of Leavitt path algebras}, 
J. Algebra, \textbf{476} (2017), 267--296.

\bibitem{AB}
A.~Bailey,
{\it Rings whose Krull dimensions are larger than their cardinalities},
Ph.D. Thesis, University of Idaho, 1999.

\bibitem{AC}
A.~Chernikov, I.~Kaplan, S.~Shelah,
{\it On non-forking spectra},
J. Eur. Math. Soc. (JEMS) \textbf{18} (2016), no. 12, 2821--2848. 

\bibitem{DD0}
D.~Dobbs,
{\it On chains of overrings of an integral domain},
Commutative rings, 95--101, Nova Sci. Publ., Hauppauge, NY, 2002.

\bibitem{DD}
D.~Dobbs,
{\it On the maximal cardinality of an infinite chain of vector subspaces},
Int. Electron. J. Algebra \textbf{13} (2013), 63--68.

\bibitem{DD2}
D.~Dobbs, R.~Heitmann,
{\it Realizing infinite cardinal numbers via maximal chains of intermediate fields},
Rocky Mountain J. Math. \textbf{44} (2014), no. 5, 1471--1503. 

\bibitem{DD3}
D.~Dobbs, B.~Mullins,
{\it On the lengths of maximal chains of intermediate fields in a field extension},
Comm. Algebra \textbf{29} (2001), no. 10, 4487--4507. 

\bibitem{DE}
D.~Eisenbud,
{\it Commutative algebra. With a view toward algebraic geometry}.
Graduate Texts in Mathematics, 150. Springer-Verlag, New York, 1995.

\bibitem{RF}
R.~Fra\"iss\'e,
{\it Theory of relations. Translated from the French}.
Studies in Logic and the Foundations of Mathematics, 118. North-Holland Publishing Co., Amsterdam, 1986.

\bibitem{RG}
R.~Gilmer,
{\it Multiplicative ideal theory. Corrected reprint of the 1972 edition}.
Queen's Papers in Pure and Applied Mathematics, 90. Queen's University, Kingston, ON, 1992.

\bibitem{Goodearl} 
K.~Goodearl, 
{\it Leavitt path algebras and direct limits}, 
Contemp. Math. \textbf{480} (2009) 165--187.

\bibitem{KG}
K.~Goodearl, R.~Warfield,
{\it An introduction to noncommutative Noetherian rings. Second edition}.
London Mathematical Society Student Texts, 61. Cambridge University Press, Cambridge, 2004.

\bibitem{JR}
R.~Gordon, J.~Robson,
{\it Krull dimension}.
Memoirs of the American Mathematical Society, No. 133. American Mathematical Society, Providence, R.I., 1973.

\bibitem{RH}
R.~Hartshorne,
{\it Algebraic geometry}.
Graduate Texts in Mathematics, 52. Springer-Verlag, New York-Heidelberg, 1977.

\bibitem{TJ}
T.~Jech,
{\it Set theory. The third millennium edition, revised and expanded}.
Springer Monographs in Mathematics. Springer-Verlag, Berlin, 2003.

\bibitem{BK}
B.~Kang, H.~Park,
{\it Krull-dimension of the power series ring over a nondiscrete valuation domain is uncountable},
J. Algebra \textbf{378} (2013), 12--21. 

\bibitem{JK}
H.J.~Keisler,
{\it Six classes of theories},
J. Austral. Math. Soc. Ser. A \textbf{21} (1976), no. 3, 257--266.

\bibitem{PK}
P.~Komj\'ath, V.~Totik,
{\it Problems and theorems in classical set theory}. 
Problem Books in Mathematics. Springer, New York, 2006.

\bibitem{Krause}
G.~Krause,
{\it On the Krull-dimension of left noetherian left Matlis-rings}.
Math. Z. \textbf{118} (1970), 207--214. 

\bibitem{MK}
M.~Kumar,
{\it Valuations and rank of ordered abelian groups}, 
Proc. Amer. Math. Soc. \textbf{133} (2004), no. 2, 343--348.

\bibitem{KK}
K.~Kunen,
{\it Set theory. An introduction to independence proofs}.
Studies in Logic and the Foundations of Mathematics, 102. North-Holland Publishing Co., Amsterdam-New York, 1980.

\bibitem{LAM}
T.Y.~Lam,
{\it A first course in concommutative rings. Second edition}.
Graduate Texts in Mathematics, 131. Springer-Verlag, New York, 2001.

\bibitem{BM}
B.~McDonald,
{\it Finite rings with identity}.
Pure and Applied Mathematics, Vol. 28. Marcel Dekker, Inc., New York, 1974.

\bibitem{WM}
W.~Mitchell,
{\it Aronszajn trees and the independence of the transfer property}, 
Ann. Math. Logic \textbf{5} (1972/73), 21--46.

\bibitem{MN}
M.~Nagata,
{\it Local rings}.
Interscience Tracts in Pure and Applied Mathematics, No. 13. Interscience Publishers (a division of John Wiley \& Sons),  New York-London, 1962.

\bibitem{GO}
G.~Oman,
{\it Small and large ideals of an associative ring}, 
J. Algebra Appl. \textbf{13} (2014), no. 5, 1350151, 20 pp.

\end{thebibliography}

\end{document}